\documentclass[11pt]{amsart}
\usepackage{hyperref}

\newtheorem{theorem}{Theorem}[section]
\newtheorem{proposition}[theorem]{Proposition}

\theoremstyle{definition}
\newtheorem{example}[theorem]{Example}

\numberwithin{equation}{section}

\DeclareMathOperator{\norm}{norm}
\newcommand{\lf}{\left\lfloor}
\newcommand{\rf}{\right\rfloor}

\begin{document}

\title{Periodicity in the $p$-adic valuation of a polynomial}

\author{Luis A. Medina}
\address{Departament of  Mathematics, University of  Puerto Rico, Rio 
Piedras, San Juan, PR 00936-8377}
\email{luis.medina17@upr.edu}

\author{Victor H. Moll}
\address{Department of Mathematics,
Tulane University, New Orleans, LA 70118}
\email{vhm@tulane.edu}

\author{Eric Rowland}
\address{University of Liege \\
D\'epartement de Math\'ematiques \\
4000 Li\`ege, Belgium}
\curraddr{
	Department of Mathematics \\
	Hofstra University \\
	Hempstead, NY
}
\email{eric.rowland@hofstra.edu}


\date{March 28, 2017}


\begin{abstract}
For a prime $p$ and an integer $x$, the $p$-adic valuation of $x$ is denoted by $\nu_{p}(x)$.  For a
polynomial $Q$ with 
integer coefficients, the sequence of valuations $\nu_{p}(Q(n))$ is shown to be either periodic or unbounded. The first 
case corresponds to the situation where $Q$ has no roots in the ring of $p$-adic integers.  In the periodic situation, the period length is determined.
\end{abstract}

\maketitle


\section{Introduction}
\label{section introduction}

For $p$ prime and $n \in \mathbb{N}$, the exponent of the highest power of $p$ that divides
$n$ is called the \textit{$p$-adic valuation of $n$}. This is denoted by 
$\nu_{p}(n)$. Given a function $f: \mathbb{N} \to \mathbb{N}$, the 
study of sequences $\nu_{p}(f(n))$ goes back to at least  Legendre~\cite{legendre-1830a}, who established the classical formula 
\begin{equation}
\nu_{p}(n!) = \sum_{k=1}^{\infty} \lf \frac{n}{p^{k}} \rf = \frac{n - s_{p}(n)}{p-1},
\end{equation}
where $s_{p}(n)$ is the sum of the digits of $n$ in base $p$. 

The work presented here  forms part of a general project 
to analyze the sequence
\begin{equation*}
V_{x} = \{ \nu_{p}(x_{n}): n \in \mathbb{N} \}
\end{equation*}
for given sequence $ x = \{ x_{n} \}$.  Valuations have been studied for the Stirling 
numbers $S(n,k)$~\cite{amdeberhan-2008b, berribeztia-2010a}, sequences satisfying first-order recurrences~\cite{amdeberhan-2009a}, the Fibonacci numbers~\cite{medinal-2015b}, the ASM (alternating sign matrices) 
numbers~\cite{beyerstedt-2011a,sunx-2009a}, and coefficients of a polynomial 
connected to a quartic integral~\cite{amdeberhan-2008a,
boros-2001c,sunx-2010a}. Other results of this type 
appear in \cite{amdeberhan-2013h,castro-2015a,cohenh-1999a,
cohn-1999a,straub-2009a}.

Consider the  sequence of valuations 
\begin{equation}
V_{p}(Q) = \{ \nu_{p}(Q(n)): n \in \mathbb{N} \},
\end{equation}
for a prime $p$ and a polynomial $Q \in \mathbb{Z}[x]$.   The polynomial $Q$ is assumed to be irreducible over 
$\mathbb{Z}$; otherwise the identity 
\begin{equation}
V_{p}(Q_{1}Q_{2}) = V_{p}(Q_{1})+ V_{p}(Q_{2})
\end{equation}
can be used to express $V_{p}(Q)$ in terms of its irreducible factors. The first result established in this paper is
 that $V_{p}(Q)$ is 
either periodic or unbounded (Theorem~\ref{theorem criteria}). In the case of a periodic sequence, the period length is explicitly determined (Theorem~\ref{theorem period rational}).  The special 
case of quadratic polynomials is discussed in detail in Sections~\ref{section quadratic2} and \ref{section quadratic odd}.

The analysis includes the $p$-adic numbers $\mathbb{Q}_{p}$ and the ring of integers $\mathbb{Z}_{p}$. Recall that 
each $x \in \mathbb{Q}_{p}$ can be expressed in the form 
\begin{equation}
x = \sum_{k=k_{0}}^{\infty} c_{k}p^{k}
\end{equation}
with $0 \leq c_{k} \leq p - 1$ and $c_{k_{0}} \neq 0$.

The $p$-adic integers $\mathbb{Z}_{p}$ correspond to the 
case $k_{0} \geq 0$, and invertible elements in this ring have $k_{0} = 0$. The set of invertible elements is denoted by $\mathbb{Z}_{p}^{\times}$.
The $p$-adic absolute value of $x \in \mathbb{Q}_{p}$ is defined by $|x|_{p} = p^{-k_{0}}$. In particular, $x \in 
\mathbb{Z}_{p}^{\times}$ if and only if $x \in \mathbb{Z}_{p}$ and $|x|_{p} = 1$.

The determination of the sequence $V_{p}(Q)$ will require examining the irreducibility of $Q$ in $\mathbb{Z}_{p}[x]$. Some 
classical criteria are stated below.

\begin{theorem}[{Eisenstein criterion~\cite[Proposition~$5.3.11$]{gouvea-1997a}}]
Let $f(x) = a_{n}x^{n} + \cdots + a_{1}x + a_{0} \in \mathbb{Z}_{p}[x]$. Assume 
\begin{enumerate}
\item
$\nu_{p}(a_{n}) = 0$, 
\item
$\nu_{p}(a_{j}) > 0 \text{ for } 0 \leq j < n$, and
\item
$\nu_{p}(a_{0}) = 1$.
\end{enumerate}
Then $f$ is irreducible in $\mathbb{Z}_{p}[x]$.
\end{theorem}

\begin{theorem}[{Hensel's lemma, polynomial version~\cite[Theorem~$3.4.6$]{gouvea-1997a}}]
Let $f \in \mathbb{Z}_{p}[x]$ and assume there are non-constant 
polynomials $g, \, h \in \mathbb{Z}_{p}[x]$, such that 
\begin{enumerate}
\item
$g$ is monic, 
\item
$g$ and $h$ are coprime modulo $p$, and
\item
$f(x) \equiv g(x) h(x) \mod p$. 
\end{enumerate} 
Then $f$ is reducible in $\mathbb{Z}_{p}[x]$.
\end{theorem}

\begin{theorem}[Dumas Irreducibility Criterion~\cite{dumas-1906a}]
Let $f \in \mathbb{Z}_{p}[x]$ be given by 
\begin{equation}
f(x)=a_0 x^n+a_{1}x^{n-1}+\cdots+a_{n-1}x+a_n. 
\end{equation}
 Suppose that
\begin{enumerate}
\item $\nu_p(a_0)=0$,
\item $\nu_p(a_i)/i > \nu_p(a_n)/n$ for $1\leq i \leq n-1$, and 
\item $\gcd(\nu_p(a_n),n)=1$.
\end{enumerate}
Then $f$ is irreducible in $\mathbb{Z}_p[x]$.
\end{theorem}


\section{Boundedness of the sequence $V_{p}(Q)$.}
\label{section boundedness}

This section characterizes the boundedness of the sequence $V_{p}(Q)$ in terms of the existence of zeros of the 
polynomial $Q$ in $\mathbb{Z}_{p}$.  Bell~\cite{bellj-2007a} showed that $V_{p}(Q)$ is periodic in the case that $Q$ has no zeros in $\mathbb{Z}_p$ and gave a bound for the minimal period length.

\begin{theorem}
\label{theorem criteria}
Let $p$ be a prime and $Q \in \mathbb{Z}[x]$. Then $V_{p}(Q)$ is either 
periodic or unbounded. Moreover, $V_{p}(Q)$ is periodic if and only if 
$Q$ has no zeros in $\mathbb{Z}_{p}$.  In the periodic case, the minimal period length is 
a power of $p$.
\end{theorem}

\begin{proof}
Assume that $Q$ has no zeros in $\mathbb{Z}_{p}$. If $V_{p}(Q)$ is not bounded
there exists a sequence $n_{j} \to \infty$ such that $\nu_{p}(Q(n_{j})) 
\to \infty$. The compactness of $\mathbb{Z}_{p}$ (see \cite{murtyr-2002a}) 
gives a subsequence converging to $n_{\infty} \in \mathbb{Z}_{p}$. Then 
$Q(n_{\infty})$ is divisible by arbitrary large powers of $p$, thus 
$Q(n_{\infty}) = 0$. This contradiction shows 
$V_{p}(Q)$ is bounded. In order to show $V_{p}(Q)$ is periodic, define
\begin{equation}
d = \sup \left\{ k \, : \, p^{k} \text{ divides } Q(n) \text{ for some } 
n \in \mathbb{Z} \right\}.
\end{equation}
Then $d \geq  0 $ and 
\begin{equation}
Q(n + p^{d+1}) = Q(n) + Q'(n)p^{d+1} + O(p^{d+2}).
\end{equation}
Since $\nu_{p}(Q(n)) \leq d$, it follows that 
\begin{equation}
\nu_{p} \left( Q(n + p^{d+1}) \right) 
= \nu_{p}(Q(n)),
\end{equation}
proving that $\nu_{p}(Q(n))$ is periodic. The minimal period length is a
divisor of $p^{d+1}$, thus a power of the prime $p$.

On the other hand, if $Q$ has a zero $x= \alpha$ in $\mathbb{Z}_{p}$, 
\begin{equation}
Q(x) = (x - \alpha)Q_{1}(x), \text{ with  } Q_{1} \in \mathbb{Z}_{p}[x]. 
\end{equation}
Then $\nu_{p}(Q(n)) \geq \nu_{p}(n - \alpha)$, and $V_{p}(Q)$ is unbounded.
\end{proof}

The most basic result for establishing the existence of a zero of a polynomial in $\mathbb{Z}_{p}$ is Hensel's lemma~\cite[Theorem~3.4.1]{gouvea-1997a}. In the following form, 
it states that a  simple root of a polynomial modulo $p$ has a unique lifting to a root in 
$\mathbb{Z}_{p}$.

\begin{theorem}[Hensel's lemma]
If $f \in \mathbb{Z}[x]$ and $a \in \mathbb{Z}_{p}$ satisfies 
\begin{equation}
f(a) \equiv 0 \mod p \text{ and } f'(a) \not \equiv 0 \mod p
\end{equation}
\noindent
then there is a unique $\alpha \in \mathbb{Z}_{p}$ such that 
$f(\alpha) = 0$ and $\alpha \equiv a \mod p$.
\end{theorem}

The following extension appears in \cite[Lemma~3.1]{cassels-1986a}.

\begin{proposition}
\label{proposition Hensel}
Assume  $f \in \mathbb{Z}[x]$ and $a \in \mathbb{Z}_{p}$ satisfies 
\begin{equation}
\nu_{p}(f(a)) > 2 \nu_{p}(f'(a)).
\end{equation}
Then  there is $\alpha \in \mathbb{Z}_{p}$ with $f(\alpha) = 0$ and $\alpha \equiv a \mod p$.
\end{proposition}


\section{Quadratic polynomials and the prime $p=2$}
\label{section quadratic2}

Let $a \in \mathbb{Z}$ and $Q_{a}(x) = x^{2}-a$.
This section considers the periodicity of the sequence $\{ \nu_{2}(n^{2}-a) \}$.
In view of Theorem~\ref{theorem criteria}, this is equivalent to the existence of a zero of $Q_{a}$ in $\mathbb{Z}_{2}$.
An elementary proof of Proposition~\ref{proposition bd}
appears in \cite{byrnes-2015a}. Define $c$ and $\mu(a)$ by 
\begin{equation}
a = 4^{\mu(a)} c 
\label{definition c}
\end{equation}
with $c \not \equiv 0 \mod 4$. 

\begin{proposition}
\label{proposition bd}
The polynomial $Q_{a}$ has no zeros in $\mathbb{Z}_{2}$ if and only if 
$c \not \equiv 1 \mod 8$.
\end{proposition}
\begin{proof}
Assume first that $Q_{a}$ has no zeros in $\mathbb{Z}_{2}$ and $c \equiv 1 
\mod 8$. If $a$ is odd, then $a = c = 1 + 8j$ with $j \in \mathbb{Z}$. Then 
$Q_{a}(1) = 1 - a = -8j$ and 
\begin{equation}
|Q_{a}(1)|_{2} \leq \tfrac{1}{8} \text{ and } |Q_{a}'(1)|_{2} = \tfrac{1}{2}.
\end{equation}
Therefore $|Q_{a}(1)|_{2} < \left( |Q_{a}'(1)|_{2} \right)^{2}$ and Proposition~\ref{proposition Hensel} 
produces $\alpha \in \mathbb{Z}_{2}$ with $Q_{a}(\alpha)=0$.
This is a contradiction.

 In the case $a$ even, write $a = 4^{i}(1+8j)$ with $i > 0$ and $i \in 
\mathbb{Z}$. The previous case shows the existence of 
$\alpha \in \mathbb{Z}_{2}$ with $\alpha^{2} = (1 + 8j)$. Then 
$\beta = 2^{i} \alpha$ satisfies $Q_{2}(\beta)= 0$, yielding a contradiction.

\smallskip

Assume now that $c \not \equiv 1 \mod 8$.  If $a$ is odd, then $a=c$ and 
$a \equiv 3, \, 5, \, 7 \mod 8$. A simple calculation shows that 
\begin{equation}
\nu_{2}(n^{2} - 8i-3) = \nu_{2}(n^{2}-8i-7)  = \begin{cases} 
 1 & \text{if $n$ is odd} \\
 0 & \text{if $n$ is even},
\end{cases} 
\end{equation}
and 
\begin{equation}
\nu_{2}(n^{2} - 8i-5) = \begin{cases} 
 2 & \text{if $n$ is odd} \\
 0 & \text{if $n$ is even}.
\end{cases} 
\end{equation}
For these values of $a$, the sequence $V_{2}(Q)$ is bounded. Theorem~\ref{theorem criteria} now 
shows that $Q_{a}$ has no zeros in $\mathbb{Z}_{2}$. 

If $a$ is even, then it can be written as $a = 4^{j}(8i+r)$ with $j \geq 0$ and $r = 
2, \, 3, \, 5, \, 6, \, 7$. The excluded 
case $r=4$ can be reduced to one of the residues listed above by consideration of the parity of 
the index $i$.  Now suppose $Q_{a}(x)$ has a zero $\beta \in 
\mathbb{Z}_{2}$; that is, $\beta^{2} = a = 4^{j}(8i+r)$. Then 
$\alpha = \beta/2^{j} \in \mathbb{Z}_{2}$ satisfies $\alpha^{2} = 8i+r$. 
Each of these cases lead to a contradiction. Indeed, if $r =  3, \, 5, \, 7$
the valuations $\nu_{2}(n^{2}-8i-r)$ are bounded contradicting 
Theorem~\ref{theorem criteria}. In the remaining two cases, the polynomial $x^{2}-8i-r$ 
is irreducible over $\mathbb{Z}_{2}$ by a direct application 
of the Eisenstein criterion. Therefore $Q_{a}(x)$ has no zeros. This 
concludes the proof.
\end{proof}

The previous result is now restated in terms of periodicity. The explicit period length is given in Section 
\ref{section general}.

\begin{theorem}
\label{theorem periodic-11}
Let $Q(x) = x^{2}-a$. Define $c$ by the relation $a = 4^{\mu(a)}c$, with $c \not \equiv 0 \mod 4$. Then the sequence 
$V_{2}(Q)$ is periodic if and only if $c \not \equiv 1 \mod 8$. 
\end{theorem}

Combining Theorem~\ref{theorem criteria},  Proposition~\ref{proposition bd}, and the 
classical result of Lagrange on representations of integers as sums of 
squares shows that the sequence of valuations $\{ \nu_{2}(n^{2}+b): n \in \mathbb{N} \}$ is 
bounded if and only if $b$ cannot be written as a sum of three squares.


 \section{Quadratic polynomials and an odd prime}
\label{section quadratic odd}

This section extends the results of Section~\ref{section quadratic2} to the case of odd primes.  

\begin{theorem}
\label{theorem odd prime}
Let $p \neq 2$ be a prime, and let $a \in \mathbb{Z}$ with $k = \nu_p(a)$. The
sequence $\nu_p(n^2-a)$ is periodic if and only if $k$ is odd or $a/p^k$ is a quadratic non-residue
modulo $p$. If it is periodic,  its period length is $p^{\lceil k/2 \rceil}$.
\end{theorem}

\begin{proof}
Let $p\neq 2$. Hensel's lemma shows that an integer $a$ not divisible by $p$ has a square root in $\mathbb{Z}_p$ if 
and only if $a$ is a quadratic residue modulo $p$.  This implies that $a \in  \mathbb{Q}_p$ is a square if and
 only if it can be 
written as $a = p^{2m}u^2$ with $m \in \mathbb{Z}$ and $u \in \mathbb{Z}_p^{\times}$ a 
$p$-adic unit.  Then  $x^2-a$ has a zero in 
$\mathbb{Z}_p$ is equivalent to $k$ being even and $a/p^{k}$ being a quadratic residue modulo $p$.  This proves the
 first part of the theorem.

Now assume that  $\nu_p(n^2-a)$ is periodic. It is shown that its period length is given by $p^{\lceil k/2 \rceil}$. Suppose first 
that $k$ is odd.  Let $k_{*}= (k + 1)/2$ so that $\lceil k/2\rceil=k_{*}$ and 
\begin{equation}
\nu_p((n+p^{k_{*}})^2-a)=\nu_p(n^2-a+2p^{k_{*}} n +p^{2 k_{*}}).
\end{equation}
It is shown that 
\begin{equation}
\nu_p(2p^{k_{*}} n + p^{2 k_{*}}) > \nu_p(n^2-a),
\end{equation}
which implies $\nu_p((n+p^{k_{*}})^2-a) = \nu_p(n^2-a)$. Write $n = p^{\nu_{p}(n)} n_0$ and 
$a = p^{2 k_{*}+1}a_0$. Finally, let $\gamma= \min(\nu_{p}(n),k_{*})$. Then
\begin{eqnarray}
\nu_p\left(p^{k_{*}}(2n+p^{k_{*}})\right) &\geq& k_{*}+\min(\nu_p(2n),k_{*}) \\ \nonumber
&=&k_{*}+\gamma \\ \nonumber
&>& k_{*} + \gamma +\nu_p(p^{2\nu_{p}(n)-k_{*}-\gamma}n_0^2-p^{k_{*}-1-\gamma}a_0) \\ \nonumber
&=&\nu_p(p^{2\nu_{p}(n)}n_0^2-p^{2k_{*}-1}a_0) \\ \nonumber
&=&\nu_p(n^2-a)
\end{eqnarray}
since $0>\nu_p(p^{2\nu_{p}(n)-k_{*}-\gamma}n_0^2-p^{k_{*}-1-\gamma}a_0)$. To justify this last inequality, observe that 
if $\nu_{p}(n) \geq k_{*}$ then $2\nu_{p}(n)-k_{*}-\gamma=2(\nu_{p}(n)-k_{*})\geq 0$ and $k_{*}-1-\gamma=-1<0$, and if 
 $\nu<k_{*}$ then $2\nu-k_{*}-\gamma=\nu- k_{*}<0$ and $k_{*}-1-\gamma\geq 0$. 

Suppose now that $k$ is even and $a/p^{k}$ a quadratic non-residue.  Then, there is $m \in \mathbb{N}_{0}$ and 
$a_0 \in \mathbb{Z}$ such that $a=p^{2m}a_0$ with $a_0$ a quadratic non-residue modulo $p$.  It is now shown that
\begin{equation}
 \nu_p((n+p^m)^2-a)=\nu_p(n^2-a)
 \label{rule m}
 \end{equation}
  and that $p^m$ is minimal with this property.  If $m=0$, then \eqref{rule m} becomes 
  $\nu_{p}((n+1)^{2}) = \nu_{p}(n^{2}-a)$. Both sides vanish since  $a$ is a quadratic non-residue 
  modulo $p$.  Now, for $m>0$, the statement \eqref{rule m} becomes 
\begin{eqnarray*}
(n+p^m)^2-a&=&n^2+2np^m+p^{2m}-p^{2m}a_0.
\end{eqnarray*}
The proof of \eqref{rule m} is divided into cases. In the argument given below, it is assumed that $\gcd(n,n_{0}) = 1$. \\

\noindent
\textit{Case 1:}  Suppose that $n=p^{\beta}n_0$ with $\beta,n_0\in \mathbb{Z}$ and $\beta<m$.  Observe that 
$$\nu_p(n^2-a)=\nu_p(p^{2\beta}-p^{2m}a_0)=2\beta$$
and 
$$\nu_p(2p^mn+p^{2m})=\beta+m>2\beta.$$
Then  $\nu_p((n+p^m)^2-a)=\nu_p(n^2-a)$ as claimed. \\

\noindent
\textit{Case 2:} Suppose that $n=p^mn_0$ with $n_0 \in \mathbb{Z}$.  Note that
$$\nu_p(n^2-a)=\nu_p(p^{2m}(n_0^2-a_0))=2m,$$
where the last equality follows from the fact that $p$ does not divide $n_0^2-a_0$, since $a_0$ is a quadratic non-residue modulo $p$.  On the other hand,
\begin{eqnarray*}
\nu_p((n+p^m)^2-a)&=&\nu_p(p^{2m}n_0^2+2p^{2m}n_0+p^{2m}-p^{2m}a_0)\\
&=&\nu_p(p^{2m}[n_0^2+2n_0+1-a_0])\\
&=&\nu_p(p^{2m}[(n_0+1)^2-a_0])\\
&=&2m.
\end{eqnarray*}
This gives \eqref{rule m}. \\

\noindent
\textit{Case 3:}  Finally, suppose that $n=p^{\beta}n_0$ with $\beta,n_0\in \mathbb{Z}$ and $\beta>m$.  It is easy to see
 that $\nu_p(n^2-a)=2m$. Then 
\begin{eqnarray*}
(n+p^m)-a&=&n^2+2p^mn+p^{2m}-p^{2m}a_0\\
&=&p^{2\beta}n_0^2+2p^{m+\beta}n_0+p^{2m}-p^{2m}a_0\\
&=&p^{2m}(p^{2\beta-2m}n_0^2+2p^{\beta-m}+(1-a_0)).
\end{eqnarray*}
Now $1 -a_{0} \not \equiv 0 \mod p$ since $a_{0}$ is a quadratic non-residue. Therefore  $p$ does not divide $1-a_0$ 
and \eqref{rule m}  follows. 

\smallskip

The conclusion is that $\nu_p((n+p^{\lceil k/2 \rceil})^2-a)=\nu_p(n^2-a)$ for every $n \in \mathbb{N}$. Therefore, the 
period length is a divisor of $p^{\lceil k/2 \rceil}$.  The period length cannot be smaller, since for $n=0$
$$\nu_p((n+p^i)^2-a)=\nu_p(p^{2i}-a)=2i\neq k=\nu_p(-a)=\nu_p(n^2-a).$$
This completes the proof.
\end{proof}


\section{The sequence $V_{p}(Q)$ for a general polynomial}
\label{section general}

This section extends the results described in the previous two sections to the more general case of an 
arbitrary prime $p$ and an arbitrary polynomial in $\mathbb{Z}_p[x]$.

Let $Q \in \mathbb{Z}_p[x]$.  The $p$-adic Weierstrass preparation theorem~\cite[Theorem~6.2.6]{gouvea-1997a} implies the existence of a factorization $Q(x)=p^m u Q_1(x)H(x)$ where $Q_1(x)$ is a monic polynomial with coefficients 
in $\mathbb{Z}_p$, $u\in \mathbb{Z}_p^{\times}$, $m$ is an integer, and $H(x)$ is a series that converges in 
$\mathbb{Z}_p$ with the property that $\nu_p(H(x))=0$ for every $x \in \mathbb{Z}_p$.  Therefore, 
\begin{equation}
V_p(Q) = \{\nu_p(Q(n)) : n\in \mathbb{N}\} 
\end{equation}
is a shift of $V_p(Q_1)$, showing that the general case can be reduced to the case when $Q(x)$ is a monic polynomial.

\begin{theorem}
\label{theorem period rational}
Let $Q \in \mathbb{Z}[x]$ be a monic polynomial of degree $d \geq 2$, irreducible over $\mathbb{Z}_{p}$.  Let 
$\alpha \geq 0$ be the largest non-negative integer such that $Q(x) \equiv 0 \mod p^\alpha$ for some $x \in \mathbb{Z}$. Then
$V_{p}(Q)$ is periodic with period length $p^{\lceil \alpha/d \rceil}$.
\end{theorem}

In fact a more general result, where $\mathbb{Q}$ is replaced with an arbitrary number field, can be proved in a similar way.
Let $K$ be a number field with corresponding number ring $R$.  Choose a prime ideal $\mathfrak{p}$ of 
$R$  and let $\nu:K^\times \to \mathbb{Z}$ be the corresponding valuation.  Let $K_{\mathfrak{p}}$ be the completion of 
$K$ with respect to $\nu$ and let $\mathcal{O}=\{a \in K_{\mathfrak{p}} : \nu(a)\geq 0\}$ be the closed unit ball. 
Let $\pi \in \mathcal{O}$ be a uniformizer, that is, $\nu(\pi)=1$, and relabel the valuation as $\nu_{\pi}$.

\begin{theorem}
\label{theorem period}
Let $Q\in \mathbb{Z}[x]$ be a monic polynomial of degree $d\geq 2$, irreducible over $\mathcal{O}$.  Let $\alpha \geq 0$ be 
the largest non-negative integer such that $Q(x) \equiv 0 \mod \pi^\alpha$ for some $x \in \mathbb{Z}$.  Suppose the residue field $\mathcal{O}/\pi\mathcal{O}$ has characteristic $p$, and let
$e=\nu_{\pi}(p)$ be the ramification index.  Then $\{\nu_{\pi}(Q(n)) : n \in \mathbb{N}\}$ is periodic with period 
length $p^{\lceil \alpha/(e d) \rceil}$.
\end{theorem}

The proofs of Theorems~\ref{theorem period rational} and \ref{theorem period} are based on an expression for the valuation $\nu_{\pi}(Q(n))$. 

\begin{theorem}
\label{theorem valuations}
Let $A\subseteq \mathcal{O}$ be a subring and let $Q(x)\in A[x]$ be monic polynomial of degree $d\geq 2$, irreducible over $\mathcal{O}$.  
Let  $\alpha \geq 0$ be the largest non-negative integer such that $Q(x) \equiv 0 \mod \pi^\alpha$ for some $x \in A$.  
Choose $n_0\in A$ such that $Q(n_0)\equiv 0 \mod \pi^{\alpha}$.  Then, for $n\in A$, 
\[
	\nu_{\pi}(Q(n)) =
	\begin{cases}
		d \, \nu_{\pi}(n - n_0)		& \text{if $n \not\equiv n_0 \mod \pi^{\lfloor \alpha/d \rfloor + 1}$} \\
		\alpha				& \text{if $n\equiv n_0 \mod \pi^{\lfloor \alpha/d \rfloor + 1}$}.
	\end{cases}
\]
\end{theorem}

\begin{proof}
Define the absolute value $|\cdot|_{\pi}$ on $K$ as $|a|_{\pi} = q^{-\nu_{\pi}(a)}$, where $|\mathcal{O}/\pi\mathcal{O}|=q$.  Write
\begin{equation}
Q(x)=(x-r_1)(x-r_2)\cdots(x-r_d)
\end{equation}
over a splitting field for $Q$.  Let $r=r_1$ and define $E=K_{\mathfrak{p}}(r)$.  Then $E/K_{\mathfrak{p}}$ is a field extension of degree $d$ and the $\pi$-adic absolute value extends to $E$ by
\begin{equation}
|s|_{\pi}=|\norm_{E/K_{\mathfrak{p}}}(s)|_{\pi}^{1/d}.
\end{equation}
The norm of an element $s\in E$ is
\begin{equation}
\norm_{E/K_{\mathfrak{p}}}(s)=(-1)^{ml}a_0^l,
\end{equation}
where $x^m+a_{m-1}x^{m-1}+\cdots+a_1x+a_0$ is the minimal polynomial of $s$ over $K_{\mathfrak{p}}$ and $l$ is the degree of the extension of $E/K_{\mathfrak{p}}(s)$.

For every element $n \in A$, the minimal polynomial of $n - r$ is 
\begin{equation}
(x-(n - r_1))(x-(n - r_2))\cdots (x - (n - r_d)).
\end{equation}
Therefore
\begin{eqnarray*}
|n - r|_{\pi}&=& |(n - r_1)\cdots(n - r_d)|_{\pi}^{1/d}\\
&=&|Q(n)|_{\pi}^{1/d}\\
&=&\left(q^{-\nu_{\pi}(Q(n))}\right)^{1/d}.
\end{eqnarray*}
This gives
\begin{equation}
\nu_{\pi}(Q(n))=-d\log_q|n - r|_{\pi}
\end{equation}
(where $\log_{q}$ is the real logarithm to base $q$).  Now take any $n_0\in A$ such that $Q(n_0)\equiv 0 \mod q^{\alpha}$.  Then
\begin{eqnarray*}
|n - r|_{\pi} &\leq& \max\left(|n - n_0|_{\pi},|n_0-r|_{\pi}\right)\\
&=& \max\left(|n - n_0|_{\pi},q^{-\nu_{\pi}(Q(n_0))/d}\right)\\
&=& \max\left(|n - n_0|_{\pi},q^{-\alpha/d}\right)
\end{eqnarray*}
with equality if $|n - n_0|_{\pi} \neq q^{-\alpha/d}$.  The computation of $\nu_{\pi}(Q(n))$ from this equation is divided into three cases.
Define $\beta = \lfloor \alpha/d \rfloor$.

\medskip

\noindent
{\em Case 1:}   If $n\equiv n_0 \mod \pi^{\beta + 1}$, then $\nu_{\pi}(n - n_0)\geq \beta + 1 > \alpha/d$, and it follows that
\begin{equation}
|n - n_0|_{\pi} = q^{-\nu_{\pi}(n - n_0)}<q^{-\alpha/d}.
\end{equation}
Then $|n - r|_{\pi}=q^{-\alpha/d}$ and
\begin{equation}
\nu_{\pi}(Q(n))=-d\log_{q}|n - r|_{\pi} = \alpha.
\end{equation}

\medskip

\noindent
{\em Case 2:}  If $n\not\equiv n_0 \mod \pi^{\beta}$, then $\nu_{\pi}(n - n_0) < \beta \leq \alpha/d$, and
\begin{equation}
|n - n_0|_{\pi} = q^{-\nu_{\pi}(n - n_0)}>q^{-\alpha/d}.
\end{equation}
In this case, $|n - r|_{\pi}=|n - n_0|_{\pi}$ and
\begin{equation}
\nu_{\pi}(Q(n))=-d\log_q|n - r|_{\pi} = d \, \nu_{\pi}(n - n_0),
\end{equation}
as claimed.

\medskip

\noindent
{\em Case 3:}  The final case is $n\equiv n_0 \mod \pi^{\beta}$ and $n\not\equiv n_0 \mod \pi^{\beta + 1}$.
Then $\nu_{\pi}(n - n_0)=\beta$ and
\begin{equation}
|n - n_0|_{\pi} = q^{-\nu_{\pi}(n - n_0)} = q^{-\beta}.
\end{equation}
If $\alpha/d$ is not an integer, this implies that $|n - n_0|_{\pi}>q^{-\alpha/d}$, so that $|n - r|_{\pi} = |n - n_0|_{\pi}$ as in Case~2, and
\begin{equation}
\nu_{\pi}(Q(n)) =-d\log_q|n - r|_{\pi} = d \, \nu_{\pi}(n - n_0).
\end{equation}
On the other hand, if $\alpha/d$ is an integer then $|n - n_0|_{\pi}=q^{-\alpha/d}$.  In this case, $|n - r|_{\pi}\leq q^{-\alpha/d}$ and
\begin{equation}
\nu_{\pi}(Q(n)) = -d\log_{q}|n - r|_{\pi} \geq \alpha.
\end{equation}
Since $\nu_{\pi}(Q(n))\leq \alpha$ for all $n\in A$, it follows that $\nu_{\pi}(Q(n))=\alpha = d \, \nu_{\pi}(n - n_0)$, as claimed.
\end{proof}

The proof of Theorem~\ref{theorem period rational} is presented next, followed by the proof of Theorem~\ref{theorem period}.

\begin{proof}[Proof of Theorem~\ref{theorem period rational}]
Let $n_{0} \in \mathbb{Z}$ with $Q(n_{0}) \equiv 0 \mod p^{\alpha}$ and define 
$\beta = \lfloor \alpha/d \rfloor$. Assume first that $\alpha/d \not \in \mathbb{Z}$.
In the rational case, i.e.\ $\pi=p$, Theorem~\ref{theorem valuations} shows that 
$\nu_{p}(Q(n))$ depends only on the residue of $n$ modulo 
$p^{\beta+1}$. Therefore the period length of $V_{p}(Q)$ is at most 
$p^{\beta +1}$. Since $\alpha/d$ is not an integer and 
\begin{equation}
\nu_{p}(Q( n_{0}+p^{\beta})) = d \, \nu_p(p^\beta) =
d \, \beta \neq \alpha = \nu_{p}(Q(n_{0})),
\end{equation}
the period length is not $p^{\beta}$; therefore the period length is $p^{\beta + 1} = p^{\lceil \alpha/d \rceil}$ as claimed.

In the case $\alpha/d \in \mathbb{Z}$ (equal to $\beta$),
Theorem~\ref{theorem valuations} gives 
\[
	\nu_p(Q(n)) =
	\begin{cases}
		d \, \nu_p(n - n_0)	& \text{if $n \not\equiv n_0 \mod p^{\beta + 1}$} \\
		\alpha			& \text{if $n\equiv n_0 \mod p^{\beta + 1}$}.
	\end{cases}
\]
If $n \equiv n_0 \mod p^\beta$ and $n \not\equiv n_0 \mod p^{\beta + 1}$, then $d \, \nu_p(n - n_0) = \alpha$ and one can move this case to obtain
\[
	\nu_p(Q(n)) =
	\begin{cases}
		d \, \nu_{p}(n - n_{0}) & \text{ if } n \not \equiv n_{0} \mod p^{\beta}  \\
		\alpha & \text{ if } n  \equiv n_{0} \mod p^{\beta}.
	\end{cases}
\]
It follows that  the period length of $V_{p}(Q)$ is at most $p^{\beta}$. Since
\begin{equation*}
\nu_{p}(Q(n_{0} + p^{\beta-1})) = d \, \nu_{p}( p^{\beta-1}) = 
d \, ( \beta -1) \neq \alpha = \nu_{p}(Q(n_{0})),
\end{equation*}
the period length is not $p^{\beta - 1}$; therefore the period length is $p^\beta = p^{\lceil \alpha/d \rceil}$.
\end{proof}

The general case follows from the proof of Theorem~\ref{theorem period rational}.

\begin{proof}[Proof of Theorem~\ref{theorem period}]
Let $m$ be an integer.   Since $\mathfrak{p}\cap \mathbb{Z} = p\mathbb{Z}$, then  $\pi$ divides $Q(m)$ if and only if $p$ divides $Q(m)$ 
(since the coefficients of $Q(x)$ are rational integers). Therefore the sequences $\{\nu_{\pi}(Q(n)) : n \in \mathbb{N}\}$ and $\{\nu_{p}(Q(n)) : n \in \mathbb{N}\}$ have the same period length.  Let $\alpha_1$ be the largest non-negative integer such that $Q(x)\equiv 0 \mod p^{\alpha_1}$ has a solution in $\mathbb{Z}$.
Similarly, let $\alpha$ be the largest non-negative integer such that $Q(x)\equiv 0 \mod \pi^\alpha$ has a solution in $\mathbb{Z}$.  From the proof of Theorem~\ref{theorem period rational} one knows that the period length of these sequences is $p^{\lceil \alpha_1/d \rceil}$.  However, it follows from $e=\nu_{\pi}(p)$ 
that $\alpha = e \alpha_1$. Thus, the period length is given by
\begin{equation}
p^{\lceil \alpha_1/d \rceil}=p^{\lceil \alpha/(e d) \rceil}.
\end{equation}
This concludes the proof.
\end{proof}


\section{A collection of examples}
\label{section examples}

This final section presents some examples that illustrate the preceding theorems.

\begin{example}
Let $Q(x)=x^3+9 x^2+81 x+243$ and $p = 3$.  Dumas' criterion shows that $Q(x)$ is irreducible 
over $\mathbb{Z}_3$.  A direct calculation yields  $\alpha = 5$, i.e.\ $Q(x)\equiv 0 \mod 3^5$ has solutions, but 
$Q(x)\equiv 0 \mod 3^6$ does not.  Theorem~\ref{theorem period rational} implies that $V_3(Q)$ is periodic with period length $9$.
The explicit $3$-adic valuation of $Q(n)$ for $n\in \mathbb{Z}$ is provided by Theorem~\ref{theorem valuations}.  In this case, 
$\beta=1$ and choosing $n_0 = 0$ gives 
$$\nu_3(Q(n))=\begin{cases}
 0 & \text{if } n \not\equiv 0 \mod 3 \\
 3 &  \text{if } n \equiv 3,6 \mod 9 \\
 5 & \text{if } n \equiv 0 \mod 9.
\end{cases}
$$
Therefore the fundamental period of $V_3(Q)$ is given by $5, 0, 0, 3, 0, 0, 3, 0, 0$.
\end{example}

The next example offers an interesting twist, using the periodicity of $V_{p}(Q)$ to determine the reducibility of a 
polynomial $Q$.

\begin{example}
Take $Q(x) = x^{4} + x^{3} + x^{2} + 3x+3 \in \mathbb{Z}_{3}[x]$ and check $\alpha = 3$.
Suppose $Q$ is irreducible in $\mathbb{Z}_{3}[x]$.
Theorem~\ref{theorem period rational} then implies that $V_{3}(Q)$ is periodic with period length $3$. But $V_{3}(Q) = \{ 1, 2, 0, 1, 3, 0, \dots \}$ does not have period length $3$; this contradicts the assumption, and therefore $Q$ is reducible.  Now $Q(x) \equiv x^{2} (x+2)^{2} \mod 3$
and Hensel's lemma implies that $Q$ factors in the form 
\begin{equation}
Q(x) = (x^{2} + \gamma_{1}x+ \gamma_{0}) (x^{2} + \beta_{1}x+ \beta_{0})
\end{equation}
with $\gamma_{j}, \, \beta_{j} \in \mathbb{Z}_{3}$.  The polynomials are chosen so that 
\begin{equation}
x^{2} + \gamma_{1}x + \gamma_{2} \equiv x^{2} \mod 3  \text{ and }
x^{2} + \beta_{1}x + \beta_{2} \equiv x^{2}  + x + 1 \mod 3.
\end{equation}
A direct application of Hensel's lemma gives the expansions 
\begin{eqnarray*}
\gamma_{0} & = & p + p^{2} + p^{3} + 2p^{4} + 2p^{7} + 2p^{9} + \cdots \\
\gamma_{1} & = & 2p^{2} + 2p^{3} + p^{4} + p^{7} + 2p^{8} + \cdots \\
\beta_{0} & = & 1 + 2p + 2p^{2} + p^{3} + 2p^{4} + p^{5} + p^{6} + p^{7} + \cdots \\
\beta_{1} & = & 1 + p^{2} + p^{4} + 2p^{5} + 2p^{6} + p^{7} + 2p^{9} + \cdots,
\end{eqnarray*}
with $p=3$. The reader can now check that $V_{3}(Q_{1})$ has  period length $3$ and $V_{3}(Q_{2})$ has period length $9$.  It follows 
that $V_{3}(Q)$ is periodic with period length $9$, and the fundamental period is $1, 2, 0, 1, 3, 0, 1, 2, 0$.
\end{example}

The final examples show how Theorem~\ref{theorem period} and Theorem~\ref{theorem valuations} work in the more general setting of a number field.

\begin{example}
Consider the number field $K=\mathbb{Q}(\sqrt[3]{2})$ and its ring of integers $\mathbb{Z}[\sqrt[3]{2}]$.  Choose the prime ideal 
$\mathfrak{p}=(\sqrt[3]{2})$ and $\pi=\sqrt[3]{2}$.   Observe that $\pi$ lies above the rational prime 2. Let $Q(x)=x^2-384=x^2 - 2^7\cdot3$.  
In this case, $\alpha=21$, $\beta=10$, and $n_0=0$.  Theorem~\ref{theorem valuations} implies that
\begin{equation}
\label{exgenval}
\nu_{\pi}(Q(n)) = \begin{cases}
 2 \, \nu_{\pi}(n) & \text{if } n\not\equiv 0 \mod \pi^{11} \\
 21 & \text{if } n\equiv 0 \mod \pi^{11}.
\end{cases}
\end{equation}
Using the fact that $e=\nu_{\pi}(2)=3$, then equation~\eqref{exgenval} can be written as
\begin{equation}
\nu_{\pi}(Q(n)) = \begin{cases}
0 & \text{if } n\equiv 1 \mod 2 \\
6 & \text{if } n\equiv 2 \mod 4 \\
12 & \text{if } n\equiv 4 \mod 8 \\
18 & \text{if } n\equiv 8 \mod 16 \\
21 & \text{if } n\equiv 0 \mod 16.
\end{cases}
\end{equation}
Finally, Theorem~\ref{theorem period} implies that the period length is given by $2^{\lceil \alpha/(e d) \rceil}=2^{\lceil 21/6\rceil} = 16$.  
Indeed, the fundamental period of this sequence is
$$21, 0, 6, 0, 12, 0, 6, 0, 18, 0, 6, 0, 12, 0, 6, 0.$$
The fundamental period of $V_2(Q)$ is
\[
	7, 0, 2, 0, 4, 0, 2, 0, 6, 0, 2, 0, 4, 0, 2, 0.
\]
\end{example}

The next example deals with the case when the values of $n$ are chosen from a subring $A$ different from $\mathbb{Z}$.

\begin{example}
Consider the same number field $K=\mathbb{Q}(\sqrt[3]{2})$, but now choose the prime ideal $\mathfrak{p}=(1+\sqrt[3]{2})$ with 
uniformizer $\pi=1+\sqrt[3]{2}$.   Note that
$$\pi^3 = 3+3 \sqrt[3]{2}+3(\sqrt[3]{2})^{2} =3(1+\sqrt[3]{2}+\sqrt[3]{4}).$$
Let $u=1+\sqrt[3]{2}+\sqrt[3]{4}$ and $v=\sqrt[3]{2}-1$ and observe that 
$$uv=(1+\sqrt[3]{2}+\sqrt[3]{4})(\sqrt[3]{2}-1)=1,$$
thus $u$ and $v$ are units.  This implies that $3=v\pi^3$ and therefore $\mathcal{O}/\pi\mathcal{O} = \mathbb{F}_3$.

Consider the polynomial $Q(x)=x^2 - \pi^5 \in (\mathbb{Z}[\sqrt[3]{2}])[x]$.  
Observe that in this case, the subring $A\subseteq \mathcal{O}$ from Theorem~\ref{theorem valuations} is $A=\mathbb{Z}[\sqrt[3]{2}]$.  For this polynomial, 
$\alpha=5, \beta=2$, and $n_0=0$.  Theorem~\ref{theorem valuations} implies
\begin{equation}
\nu_{\pi}(Q(n)) = \begin{cases}
2 \, \nu_{\pi}(n) & \text{if } n\not\equiv 0 \mod \pi^{3} \\
 5 & \text{if } n\equiv 0 \mod \pi^{3}.
\end{cases}
\end{equation}
This equation can be simplified to 
\begin{equation}
\nu_{\pi}(Q(n)) = \begin{cases}
0& \text{if } n\equiv 1,2  \mod \pi \\
2& \text{if } n\equiv \pi,2\pi \mod \pi^2\\
 4 & \text{if } n\equiv \pi^2,2\pi^2 \mod \pi^{3} \\
 5 & \text{if } n\equiv 0  \mod \pi^{3}.
\end{cases}
\end{equation}
If $n$ is restricted to the subring $\mathbb{Z}\subseteq A$, then
\begin{equation}
\nu_{\pi}(Q(n)) = \begin{cases}
0& \text{if } n\equiv 1,2  \mod 3 \\
5 & \text{if } n\equiv 0  \mod 3.
\end{cases}
\end{equation}
In this case, the period length is clearly 3.
\end{example}


\section*{Acknowledgments}

The authors thank the referee for helpful suggestions.
The first author acknowledges the partial support of UPR-FIPI 1890015.00. The second 
author acknowledges the partial support of NSF-DMS 1112656.  The last author was partially supported by a Marie Curie 
Actions COFUND Fellowship.


\end{document}